\documentclass[11pt]{article}
\usepackage{latexsym}
\usepackage{amsthm}
\usepackage{amssymb}
\usepackage{amsmath}
\usepackage{rotating}
\usepackage{multirow}
\usepackage{mathrsfs}
\usepackage{wasysym}
\usepackage{esint}
\usepackage{rawfonts}
\input{prepictex}
\input{pictex}
\input{postpictex}
\DeclareMathOperator{\End}{End}

\usepackage[OT2,OT1]{fontenc}
\def\cyr{%
\renewcommand\rmdefault{wncyr}%
\renewcommand\sfdefault{wncyss}%
\renewcommand\encodingdefault{OT2}%
\normalfont
\selectfont}
\DeclareMathAlphabet{\zap}{OT1}{pzc}{m}{it}
\DeclareTextFontCommand{\textcyr}{\cyr}
\def\be{\begin{equation}}
\def\ee{\end{equation}}
\def\bea{\begin{eqnarray*}}
\def\eea{\end{eqnarray*}}

\def\CC{\mathbb C}

\newtheorem{main}{Theorem}

\DeclareMathOperator{\Vol}{Vol}

\newtheorem{thm}{Theorem}
\newtheorem{lem}{Lemma}
\newtheorem{prop}{Proposition}

\def\RR{{\mathbb R}}
\def\CP{{\mathbb C \mathbb P}}
\begin{document}

\title{Einstein Metrics, Harmonic Forms,\\ and Conformally K\"ahler Geometry}

\author{Claude LeBrun\\Stony Brook University}

\date{}
\maketitle

\begin{abstract}
The author has elsewhere given a complete
 classification of the compact oriented Einstein $4$-manifolds
that   satisfy $W^+(\omega , \omega )> 0$ for some self-dual harmonic $2$-form $\omega$,
where $W^+$ denotes the self-dual Weyl curvature.
In this article,      similar results  are obtained  when  $W^+(\omega , \omega )\geq  0$, 
provided  the  self-dual 
harmonic $2$-form 
$\omega$  is   transverse to the zero section of  $\Lambda^+\to M$.
 However, this  
 transversality condition   plays an 
essential role in the story; dropping it leads one into   wildly different
territory where  entirely different phenomena predominate. 
\end{abstract}

\section{Introduction}

Recall that  a Riemannian metric $h$ is said to be {\em Einstein} \cite{bes} if it has  
 constant Ricci curvature, or in other words if  it solves the Einstein equation
\begin{equation}
\label{einsteineq}
r= \lambda h
\end{equation}
for some real number $\lambda$,  where $r$ is the Ricci tensor of $h$. When this happens, 
$\lambda$ is called the {\em Einstein constant} of $h$, and of course has the same sign as  the Einstein metric's scalar curvature. 

Dimension four  seems to represent a sort of ``Goldilocks zone''  for the Einstein equation. 
In lower dimensions, Einstein metrics are extremely rigid, in the sense that 
they necessarily have constant sectional curvature, and so do not really exhibit any interesting local differential geometry. 
In higher dimensions, on the other hand, 
 they are extremely   flexible, existing in such profusion  on  familiar manifolds \cite{bohm,bgk,cvc}  that their local geometry 
 seems to offer little clue as to 
 the identity  of the manifold where they reside. 
 By contrast, dimension four seems ``just right'' for \eqref{einsteineq}, 
 as   four-dimensional Einstein metrics exhibit a well-tempered  combination of local
 flexibility and global rigidity that often makes their geometry  
 perfectly reflect  the manifold on which they live. 
 For example, if 
  $M$ is a compact real or complex-hyperbolic $4$-manifold, a $4$-torus, or $K3$, 
the  moduli space of Einstein metrics on $M$ is known explicitly, and moreover  turns out  to be
 {connected} \cite{bes,bcg,lmo}.

Unfortunately, however, we do not have a similarly complete understanding of  the moduli space of  Einstein metrics
on most of the  $4$-manifolds where  this moduli space is non-empty. An important  family of test-cases is provided  by the {\em Del Pezzo surfaces}, 
here  understood  to  mean 
the smooth compact oriented $4$-manifolds that support complex structures  with ample anti-canonical 
line bundle.   Up to diffeomorphism, there are
exactly ten such manifolds, namely  $S^2 \times S^2$ and the nine connected sums $\CP_2\# m\overline{\CP}_2$,  $m = 0, 1, \ldots, 8$. 
These 
$4$-manifolds are  completely characterized \cite{chenlebweb} by  two properties:   
they admit Einstein metrics with $\lambda > 0$, and they also admit symplectic structures. However, 
it is  currently unclear whether  the  known Einstein metrics on these spaces sweep out  the entire
Einstein moduli space. 
One of our main objectives here will be to generalize and  strengthen a  characterization of the known Einstein metrics
on Del Pezzo surfaces   previously proved by the  author in \cite{lebcake}.

In order to formulate our results, first recall that 
the bundle  $\Lambda^2\to M$ of 2-forms on  an
oriented Riemannian $4$-manifold $(M,h)$  invariantly  decomposes 
as the Whitney sum
\begin{equation} 
\Lambda^2 = \Lambda^+ \oplus \Lambda^- 
\label{deco} 
\end{equation}
 of the eigenspaces of 
the Hodge star 
operator
$\star: \Lambda^2 \to \Lambda^2$. 
Sections of the $(+1)$-eigenbundle $\Lambda^+$ are   called {self-dual 
2-forms}, while  the  
sections of the $(-1)$-eigenbundle $\Lambda^-$    are  called {anti-self-dual
2-forms}. The decomposition \eqref{deco} is moreover {\em conformally invariant},
meaning that it unchanged by multiplying the metric by an arbitrary positive function.


One  important consequence of 
the decomposition \eqref{deco} is  that it induces an invariant  decomposition of
the Riemann curvature tensor $\mathcal{R}$ into simpler pieces. 
Indeed, if  we  identify   the Riemannian curvature tensor 
with  the  self-adjoint endomorphism 
${\mathcal R} : \Lambda^2 \to \Lambda^2$
of the $2$-forms defined by 
$$\varphi_{ab} \longmapsto \frac{1}{2} {{\mathcal R}^{cd}}_{ab} \varphi_{cd}$$
and known as  the  {\em curvature operator}, then 
(\ref{deco}) allows us to decompose $\mathcal{R}$  into irreducible pieces 
\begin{equation}
\label{curv}
{\mathcal R}=
\left(
\mbox{
\begin{tabular}{c|c}
&\\
$W^++\frac{s}{12}$&$\stackrel{\circ}{r}$\\ &\\
\cline{1-2}&\\
$\stackrel{\circ}{r}$ & $W^-+\frac{s}{12}$\\&\\
\end{tabular}
} \right) ,
\end{equation}
where $s$ denotes the {scalar curvature},
$\stackrel{\circ}{r}=r-\frac{s}{4}g$ is the   
     trace-free  Ricci curvature, and where the remaining pieces 
$W^\pm$, known as the
{\em self-dual} and  {\em anti-self-dual}  Weyl tensors, 
 are the trace-free parts  of the  endomorphisms of $\Lambda^\pm$ induced by  $\mathcal{R}$.  
Remarkably enough, the corresponding pieces  ${(W^\pm)^a}_{bcd}$  of the 
Riemann  curvature tensor are both {\em conformally invariant} ---  they remain unaltered if  the
metric is multiplied by an arbitrary smooth positive function. 

Now  let  $(M,h)$ be a {\em  compact} oriented
Riemannian  $4$-manifold.
The  Hodge theorem then tells us that every deRham class
on $M$ has a unique harmonic representative. In particular,  there is a canonical isomorphism  
$$H^2(M,{\mathbb R}) =\{ \varphi \in \Gamma (\Lambda^2) ~|~
d\varphi = 0, ~ d\star \varphi =0 \} .$$
However, since the Hodge star operator $\star$ defines an involution of 
the right-hand side, we obtain 
 a direct-sum decomposition
\begin{equation}
	H^2(M, {\mathbb R}) = {\mathcal H}^+_{h}\oplus {\mathcal H}^-_{h},
	\label{harm}
\end{equation}
where
$${\mathcal H}^\pm_{h}= \{ \varphi \in \Gamma (\Lambda^\pm) ~|~
d\varphi = 0\} $$
are the spaces of self-dual and anti-self-dual harmonic forms.
Since the conditions of being closed and belonging to $\Lambda^\pm$  are both 
conformally invariant, it follows that  the spaces ${\mathcal H}^\pm$ are 
both conformally invariant, too. 
Moreover, the dimensions $b_\pm = \dim \mathcal{H}^\pm$ of these spaces are completely
metric-independent,  and  can easily be  shown to  be oriented homotopy invariants 
of the $4$-manifold $M$. 

%
%

Now, if $(M,h)$ is a compact oriented Riemannian $4$-manifold, and if $\omega\in \mathcal{H}^+$ is a fixed
self-dual harmonic $2$-form, the quantity
$$W^+(\omega , \omega ) := \langle W^+(\omega ) , \omega \rangle = \frac{1}{4} (W^+)^{abcd} \omega_{ab} \omega_{cd}$$
transforms in an extremely simple manner under conformal rescaling; namely, if we change our metric 
by 
$$h\rightsquigarrow u^2h$$
for some positive function $u$, then the quantity in question changes by 
$$
W^+(\omega , \omega )\rightsquigarrow u^{-6}W^+(\omega , \omega ).
$$
In particular, the {\em sign} of this quantity at a given point is unchanged by 
conformal rescalings. This makes this hybrid measure of curvature particularly compelling
when $b_+(M)=1$, because in this case there is, up to a non-zero constant factor,  only one non-trivial choice of $\omega$,
and the sign of $W^+(\omega , \omega )$ at each point then becomes a natural global conformal invariant of $(M,h)$.

The main result of \cite{lebcake} was that if a compact $4$-dimensional Einstein manifold satisfies
\begin{equation}
\label{delphi}
W^+(\omega , \omega ) > 0
\end{equation}
for some self-dual harmonic $2$-form $\omega$, then $(M,h)$ is one of the known Einstein metrics
on some Del Pezzo surface. Conversely, the known Einstein metrics on Del Pezzo surfaces all 
have this property. Combining these two observations then shows, as a corollary, that the known 
Einstein metrics on these spaces always sweep out a connected component of the moduli space.
Here it is worth noting that every Del Pezzo surface has $b_+=1$, so that condition \eqref{delphi} 
represents a rather natural characterization of the known Einstein metrics on these
$4$-manifolds. 

On the other hand, since condition \eqref{delphi} trivially  implies that both $W^+$ and $\omega$ are nowhere zero,
it might seem desirable to relax this overly-stringent condition by merely requiring that $W^+(\omega, \omega )$
be non-negative. What we will show here is that this can indeed be done, provided one imposes
an interesting and  natural condition on  the $2$-form. Namely,  if $\omega$ is a harmonic self-dual $2$-form on 
a compact oriented Riemannian $4$-manifold  $(M,h)$, one says that  $\omega$ is  {\em near-symplectic} 
if its graph  is  transverse to the zero section of the 
 rank-$3$ vector  bundle $\Lambda^+\to M$.  This is a {\em generic} condition, as  has 
 come to be understood through the work of 
  Taubes \cite{taubesnear1,taubesnear} and others  \cite{koho,lcp2,perutz}; 
 indeed, on any smooth compact oriented $4$-manifold with $b_+\neq 0$, the set of metrics admitting  a
 near-symplectic self-dual harmonic $2$-form is  open and dense.
 Of course, a dimension count immediately reveals  that the zero locus of a near-symplectic self-dual harmonic $2$-form  
$\omega$ on $(M,h)$ is automatically a (possibly empty) finite disjoint union $Z$ of circles:
\begin{equation}
\label{loco}
Z\approx \sqcup _{j=1}^n S^1.
\end{equation}
Imposing this reasonable assumption on  the behavior of  $\omega$  will  actually allow us to prove some natural generalizations of
the main result of \cite{lebcake}. 
More specifically, here are the  main results of the present article: 

\begin{main} 
\label{diamond}
Let $(M,h)$ be a compact oriented Einstein $4$-manifold that  carries  a near-symplectic self-dual harmonic $2$-form $\omega$ such that 
\begin{equation}
\label{rosette}
W^+(\omega , \omega ) \geq 0, \qquad  W^+(\omega , \omega )\not\equiv 0.
\end{equation}
Then $W^+(\omega , \omega )>0$ everywhere,  $M$ is diffeomorphic to a Del Pezzo surface, and $h$ is conformally related to 
a positive-scalar-curvature  extremal  K\"ahler metric $g$ on $M$ with K\"ahler form $\omega$.  
Conversely, every Del Pezzo surface admits an Einstein metric $h$ satisfying \eqref{rosette}  for a   self-dual harmonic $2$-form $\omega$
that is nowhere zero (and hence near-symplectic).
\end{main}
 
 \medskip
 
\begin{main} \label{pearl}
Let $(M,h)$ be a compact oriented $\lambda \geq 0$ Einstein $4$-manifold   that carries  a near-symplectic self-dual harmonic $2$-form $\omega$ such that 
\begin{equation}
\label{cosette}
W^+(\omega , \omega ) \geq 0
\end{equation}
everywhere.
Then $\omega$ is nowhere zero,   and $h$ is conformally related to 
an extremal  K\"ahler metric $g$ on $M$  with K\"ahler form $\omega$.  Moreover, $M$ is diffeomorphic to a 
Del Pezzo surface, a $K3$ surface, an Enriques surface, an Abelian surface, or a hyper-elliptic surface. 
Conversely, each of these  complex surfaces admits a  $\lambda \geq 0$   Einstein metric $h$   satisfying \eqref{cosette} for a   self-dual harmonic $2$-form $\omega$
that is nowhere zero (and hence near-symplectic).
\end{main}

 \medskip
  
\begin{main} \label{opal}
The  near-symplectic hypothesis in Theorem \ref{diamond} is essential:    counter-examples 
show that the result fails without this assumption. 
\end{main}

 \medskip
 
 The proofs of these main results can be found  \S\ref{demedici} below, following the proofs, in  \S\S\ref{diana}--\ref{demilo}, of  the     technical results 
that  underpin  these theorems.

\section{An Integral Weitzenb\"ock Formula}
\label{diana}

Let $(M,h)$ be a compact oriented Riemannian $4$-manifold with 
{\em harmonic self-dual Weyl curvature}, in the sense that $\delta W^+:= -\nabla \cdot W^+=0$.
 When $h$ is Einstein, this property automatically holds,  by virtue of  of the second Bianchi identity. 
 We  will further   assume throughout   that $h$ is at least $C^4$. The latter assumption is of course  innocuous in the Einstein case,
as  elliptic regularity for \eqref{einsteineq} implies   that Einstein metrics  are always \cite{det-kaz}  real-analytic in harmonic coordinates.

\pagebreak

We  will henceforth also assume that $b_+(M)\neq 0$. This is equivalent to saying that $(M,h)$ 
 admits a self-dual harmonic $2$-form $\omega\not\equiv 0$. We now choose some such 
 form, and regard it as fixed for the remainder of the discussion.  Let $Z\subset M$ denote the zero set of $\omega$.
 Since $\omega$ is self-dual by assumption, 
 $$\omega \wedge \omega = \omega \wedge \star \omega = |\omega|^2_h d\mu_h,$$
 and it therefore follows  that 
$\omega$ is actually  a symplectic form  on the open set  $X:=M-Z$ where $\omega$ is non-zero. 
Moreover, the Riemannian metric  $g$ on $X$ defined by 
 $g={2}^{-1/2}|\omega|_h h$ is then an {\em almost-K\"ahler metric}, in the sense that $g$ is related to 
 the symplectic form $\omega$ by $g= \omega (\cdot , J \cdot )$ for a unique almost-complex structure $J$ on $X$.

Let us now re-express the relationship between the conformal relationship between our two metrics as 
$$h=f^2g,$$
where  $f=2^{1/4}|\omega|_h^{-1/2}$. The fact that  $h$ satisfies $\delta W^+=0$ then implies \cite{pr2} that 
$g$ satisfies $\delta (fW^+)=0$. Since our assumptions imply that $g$  
 is also at least $C^4$, we therefore have   \cite{derd,G1,lebcake,pr2} the Weitzenb\"ock formula
\begin{equation}
\label{initio}
0 = \nabla^*\nabla (fW^+)+ \frac{s}{2} fW^+ - 6 fW^+\circ W^+ + 2 f|W^+|^2 I 
\end{equation}
for $fW^+$, which for notational simplicity has been represented here as  a trace-free section of $\End (\Lambda^+)$,
while  $s$ and $\nabla$  respectively  denote the scalar curvature 
and Levi-Civita connection of our almost-K\"ahler metric  $g$ on $X$.

Our strategy  is now to contract \eqref{initio}  with $\omega\otimes \omega$, integrate on $X=M-Z$,
and then try to integrate by parts in order to throw the Bochner Laplacian  $\nabla^*\nabla$ onto $\omega\otimes \omega$.
In order to accomplish this, we first exhaust $X$ by domains $X_\epsilon$ with smooth boundary, 
where $X_\epsilon$ is the region where $|\omega|_h\geq \epsilon$, where 
$\epsilon > 0$ is any regular value of the smooth non-negative function $|\omega|_h : X\to \RR$.
Integrating  by parts on $X_\epsilon$ then has the following effect:

\begin{lem}
\label{ready}
There is a constant $C$, independent of $\epsilon \in (0,1)$,  but depending  on $(M,h,\omega)$,  such that 
$$\left|
\int_{X_\epsilon} \left[ \langle \nabla^*\nabla (fW^+) , \omega \otimes \omega  \rangle - \langle fW^+,    \nabla^*\nabla (\omega \otimes \omega ) \rangle 
\right] d\mu_g
\right|
\leq C \epsilon^{-3/2} \Vol^{(3)} ( \partial X_\epsilon , h),
$$
where all terms in the integral on the left are computed with respect to $g$, but where 
the $3$-dimensional boundary  volume on  the right is computed with respect to  $h$.
\end{lem}
\begin{proof} By the divergence version of Stokes' theorem, we  have
\begin{eqnarray*}  \int_{X_\epsilon} \langle \nabla^*\nabla (fW^+) , \omega \otimes
\omega \rangle d\mu_g  &=&  \int_{X_\epsilon} \langle -\nabla\cdot \nabla fW^+  , \omega \otimes \omega \rangle d\mu_g
\\ &=& - \int_{X_\epsilon}  \nabla\cdot\langle  \nabla fW^+  , \omega \otimes \omega \rangle  d\mu_g
\\ && \qquad +  \int_{X_\epsilon}\langle \nabla fW^+  , \nabla (\omega \otimes \omega ) \rangle d\mu_g
\\ &=& -  \int_{\partial X_\epsilon} \langle  \nabla _\nu  fW^+  , \omega \otimes \omega \rangle  d{\zap a}_g 
\\ && \qquad +  \int_{X_\epsilon}\langle \nabla fW^+  , \nabla (\omega \otimes \omega ) \rangle d\mu_g  
\\ &=& -  \int_{\partial X_\epsilon}  \nabla _\nu \langle   fW^+  , \omega \otimes \omega \rangle  d{\zap a}_g \\ && \qquad +  \int_{\partial X_\epsilon} \langle    fW^+  , \nabla _\nu (\omega \otimes \omega )\rangle  d{\zap a}_g
 \\ && \qquad \qquad +  \int_{X_\epsilon}  \nabla\cdot  \langle  fW^+  , \nabla (\omega \otimes \omega ) \rangle d\mu_g 
 \\ && \qquad \qquad \qquad +  \int_{X_\epsilon}  \langle  fW^+  ,  -\nabla\cdot \nabla (\omega \otimes \omega ) \rangle d\mu_g 
 \\ &=& -  \int_{\partial X_\epsilon}  \nabla _\nu \langle   fW^+  , \omega \otimes \omega \rangle  d{\zap a}_g 
 \\&& \qquad + 2 \int_{\partial X_\epsilon} \langle    fW^+  , \nabla _\nu (\omega \otimes \omega )\rangle  d{\zap a}_g
  \\ && \qquad\qquad+ \int_{X_\epsilon}  \langle  fW^+  ,  \nabla^\ast \nabla (\omega \otimes \omega ) \rangle d\mu_g  
  \\ &=& -  \int_{\partial X_\epsilon}  \nabla _\nu [  fW^+  (\omega , \omega )]  d{\zap a}_g \\&& \qquad  + 4 \int_{\partial X_\epsilon}  fW^+  (\omega, \nabla _\nu \omega) d{\zap a}_g
  \\ && \qquad\qquad+ \int_{X_\epsilon}  \langle  fW^+  ,  \nabla^\ast \nabla (\omega \otimes \omega ) \rangle d\mu_g  
\end{eqnarray*} 
where $\nu$ is the outward-pointing unit normal of $\partial X_\epsilon$ with respect to $g$, and where $d{\zap a}_g$ is the $g$-induced  volume $3$-form 
 on the boundary. Here, every term is thus understood to be computed with respect to $g$.

We now estimate the boundary integrals by first re-expressing them in terms of the original metric  $h=f^2g$. For  emphasis  and clarity, we will  temporarily
use  $\hat{\nu} = f^{-1}\nu$  to denote the unit normal of $\partial X_\epsilon$
with respect to $h$, and  $\hat{\nabla}$ to denote the Levi-Civita connection of $h$, which  differs from the Levi-Civita connection of $\nabla$ of $g$ by 
$$\delta^a_b\beta_c+\delta_c^a\beta_b - \beta_d h^{da} h_{bc},$$
where $\beta = d\log f= -\frac{1}{2}d\log |\omega|_h$.
In other cases where the meaning of a term depends on a choice of metric, we will indicate the metric used by means of a
subscript; for example, since 
 index-raising is needed to define $W^+(\omega, \omega )$, one has 
$$[W^+(\omega , \omega )]_g = f^6[W^+(\omega , \omega )]_h.$$
With these conventions in hand, we thus  have \begin{eqnarray*}
\left| \int_{\partial X_\epsilon}  \nabla _\nu [  fW^+  (\omega , \omega )]_g  da_g \right|&=& 
 \left|\int_{\partial X_\epsilon} f \nabla _{\hat{\nu}} [  f^{7}W^+  (\omega , \omega )]_h  f^{-3} da_h\right| \\
&\leq & 7 \left| \int_{\partial X_\epsilon}    f^{4} (\nabla _{\hat{\nu}} f)[W^+  (\omega , \omega )]_h da_h\right| 
\\&& \qquad +  \left| \int_{\partial X_\epsilon}  f^{5} \nabla _{\hat{\nu}}[W^+  (\omega , \omega )]_h   da_h\right| 
\\&\leq& 7 \left| \int_{\partial X_\epsilon}    f^{5} |f^{-1}d  f|_h |W^+|_h |\omega|_h^2 da_h\right| 
\\&& \qquad 
+   \left| \int_{\partial X_\epsilon}  f^{5} |\hat{\nabla} W^+|_h  |\omega|_h^2     da_h\right| 
\\&&\qquad \qquad+2  \left| \int_{\partial X_\epsilon}  f^{5} | W^+|_h |\omega|_h |\hat{\nabla} \omega|_h     da_h\right| 
\\&=& 7 \left| \int_{\partial X_\epsilon} 2^{1/4} |\omega|_h^{-3/2}  |d|\omega|_h|_h |W^+|_h  da_h\right| 
\\&& \qquad +   \left| \int_{\partial X_\epsilon}  2^{5/4} |\omega|_h^{-1/2}  |\hat{\nabla} W^+|_h      da_h\right| 
\\&&\qquad \qquad +2  \left| \int_{\partial X_\epsilon}  2^{5/4} |\omega|_h^{-3/2}   |W^+|_h |\hat{\nabla} \omega|_h     da_h\right| 
\\&\leq & C_1 \epsilon^{-3/2} \Vol^{(3)} ( \partial X_\epsilon , h), 
\end{eqnarray*} 
where $C_1 = \sqrt[4]{2}\left[ {11}(\max_M |W^+|_h) (\max_M |\hat{\nabla}\omega|_h ) +2\max_M |\hat{\nabla} W^+|_h\right]$.
(In the last step, we have used the   Kato inequality $|d|\omega||\leq |\hat{\nabla} \omega|$,  and have remembered that $\epsilon < 1$
by hypothesis.)
Similarly, 
\begin{eqnarray*} \left| \int_{\partial X_\epsilon}  fW^+  (\omega, \nabla _\nu \omega)_g d{\zap a}_g \right| &=& 
 \left| \int_{\partial X_\epsilon}  f \cdot f^{6} W^+  (\omega, {\nabla} _{f\hat{\nu}} \omega)_h f^{-3} d{\zap a}_h \right|
\\&=&  \left| \int_{\partial X_\epsilon}  f^{5} W^+  (\omega, {\nabla} _{\hat{\nu}} \omega)_h  d{\zap a}_h \right|
\\&\leq &  2\left| \int_{\partial X_\epsilon}  f^{5} |W^+|_h  |\omega|_h |{\nabla}  \omega|_h d{\zap a}_h \right|
\\&\leq &  2\left| \int_{\partial X_\epsilon}  f^{5} |W^+|_h  |\omega|_h |\hat{\nabla}  \omega|_h d{\zap a}_h \right|
\\&& \qquad +  6\left| \int_{\partial X_\epsilon}  f^{5} |W^+|_h  |\omega|_h^2 |\beta|_h d{\zap a}_h \right|
\\&= &  2\left| \int_{\partial X_\epsilon}  f^{5} |W^+|_h  |\omega|_h |\hat{\nabla}  \omega|_h d{\zap a}_h \right|
\\&& \qquad +  3\left| \int_{\partial X_\epsilon}  f^{5} |W^+|_h  |\omega|_h |d|\omega|_h|_h d{\zap a}_h \right|
\\&\leq &  5\left| \int_{\partial X_\epsilon} 2^{5/4} |\omega|_h^{-3/2} |W^+|_h  |\hat{\nabla}  \omega|_h d{\zap a}_h \right|
\\&\leq & C_2 \epsilon^{-3/2} \Vol^{(3)} ( \partial X_\epsilon , h), 
\end{eqnarray*} 
where $C_2= 10\sqrt[4]{2}(\max_M |W^+|_h) (\max_M |\hat{\nabla}\omega|_h )$. 
Setting $C=C_1+4C_2$, and referring back to our integration-by-parts calculation, we thus see that the
claim now follows immediately from the triangle inequality. 
\end{proof}

So far, we have only assumed that $\omega$ is a non-trivial self-dual harmonic 
form on $(M,h)$. However, the information we have just gleaned becomes much more useful 
when $\omega$  happens to be near-symplectic:

\begin{lem}
\label{primed}
Let $\omega$ be a near-symplectic self-dual harmonic $2$-form on  a compact oriented
Riemannian $4$-manifold. Let $X=M-Z$ be the complement of the zero set
$Z$ of $\omega$,  set $f=2^{1/4}|\omega|_h^{-1/2}$ on $X$, and let $g=f^{-2}h$ be the almost-K\"ahler metric on 
$(X,\omega)$ obtained by conformally rescaling $h$ to make $|\omega|_g\equiv \sqrt{2}$. Then  
\begin{equation}
\label{party}
\int_{X} \langle \nabla^*\nabla (fW^+) , \omega \otimes \omega  \rangle  ~d\mu_g =  \int_{X} \langle fW^+,    \nabla^*\nabla (\omega \otimes \omega ) \rangle 
~ d\mu_g ,
\end{equation}
 where the integrands on both sides are  defined  with respect to $g$, and where both moreover  belong to $L^1$. In particular, 
 both integrals are finite, and may  be treated  either  as  improper Riemann integrals or as Lebesgue  integrals. 
\end{lem}
\begin{proof} To say that  $\omega$ is near-symplectic means, by definition, that the section $\omega$ of $\Lambda^+\to M$ is 
transverse to the zero section along its zero locus $Z\approx \sqcup_{j=1}^nS^1$. In particular, the derivative of $\omega$
along $Z$ induces an isomorphism between the normal bundle of $Z\subset M$ and the 
vector bundle $\Lambda^+|_Z\to Z$.  This moreover allows us construct a diffeomorphism between a  sufficiently small tubular neighborhood $\mathcal{U}$ of $Z$
and $Z\times B^3_\varepsilon$, where $B^3_\varepsilon \subset \RR^3$ is the standard $3$-ball of some small radius $\varepsilon$,
by  combining the nearest-point projection $\mathcal{U}\to Z$ with 
the components of $\omega$ relative to  some orthornormal framing of the the  vector bundle $\Lambda^+\to \mathcal{U}$.
(Here, we are using the fact that $\Lambda^+|_{\mathcal{U}}$ is  necessarily trivial because $\Lambda^+$
is oriented, $\mathbf{SO}(3)$ is connected, and 
$\mathcal{U}$ deform retracts to a union of circles.) 
Via this diffeomorphism, the function  $|\omega|_h$ on $\mathcal{U}$ then just becomes the standard radius function on $B^3_\varepsilon$.
Moreover,  after  reducing the size of  $\varepsilon$  if necessary, 
 the Riemannian metric $h$ on $\mathcal{U}$ becomes quasi-isometric to the standard flat product metric $h_0$
on $Z\times B^3_\varepsilon$, in the sense that $h_0/\kappa < h < \kappa h_0$ for some constant
$\kappa > 1$, and where  we have $|\omega|_h \geq \varepsilon$ on the complement $M-\mathcal{U}$ of $\mathcal{U}$. It then follows that
 the hypersurfaces $(\partial X_\epsilon , h)$ are uniformly quasi-isometric to 
$(Z\times S^2_\epsilon, h_0)$, so
 there consequently  exists a positive constant $L=4\pi\kappa$ such that 
$$Vol^{(3)} ( \partial X_\epsilon , h) < L \epsilon^2$$
for all $\epsilon \in (0,\varepsilon )$. Combining this with Lemma \ref{ready} then tells us that 
$$
\left|\int_{X_\epsilon} \left[ \langle \nabla^*\nabla (fW^+) , \omega \otimes \omega  \rangle - \langle fW^+,    \nabla^*\nabla (\omega \otimes \omega ) \rangle 
\right] d\mu_g
\right|
\leq CL \sqrt{\epsilon}$$
for all $\epsilon\in (0,\varepsilon)$.
But since the contraction of \eqref{initio} with $\omega\otimes \omega$  tells us that
$$ \langle \nabla^*\nabla  (fW^+) , \omega\otimes \omega \rangle  + \frac{s}{2}f W^+(\omega , \omega ) 
 - 6 f|W^+(\omega)|^2+ 2 f|W^+|^2 |\omega |^2 =0$$
 on $(X,g)$, it therefore follows that 
 $$\left| \int_{X_\epsilon} \Big[ \langle W^+ , \nabla^*\nabla (\omega\otimes \omega )\rangle  + \frac{s}{2} W^+(\omega , \omega ) 
 - 6 |W^+(\omega)|^2+ 2 |W^+|^2 |\omega |^2 \Big] f~ d\mu_g
 \right| \leq CL \sqrt{\epsilon}$$
 for all small $\epsilon$.  Thus 
 $$
 \lim_{\epsilon \searrow 0}  \int_{X_\epsilon} \Big[ \langle W^+ , \nabla^*\nabla (\omega\otimes \omega )\rangle  + \frac{s}{2} W^+(\omega , \omega ) 
 - 6 |W^+(\omega)|^2+ 2 |W^+|^2 |\omega |^2 \Big] f~ d\mu_g =0.
 $$
 
To prove the claim, it therefore suffices to show that both integrands in  \eqref{party} are absolutely integrable, and so  belong to  $L^1$.
To see this, first notice that 
\begin{eqnarray*}
\int_X \left|\langle \nabla^*\nabla (fW^+) , \omega \otimes \omega \rangle_g\right| d\mu_g&\leq&
2\int_X \left| \nabla^*\nabla (fW^+) \right|_g d\mu_g\\
&=&2 \int_X f^2 \left| [ \nabla \cdot  \nabla (fW^+)]_g \right|_h  f^{-4} d\mu_h\\
&\leq & 2 \int_X  \left|  \nabla^* \nabla (fW^+) \right|_h  d\mu_h \\
&& \qquad 
+ 8 \int_X  \left|  \nabla (\beta \otimes fW^+) \right|_h  d\mu_h \\
&& \qquad\qquad   + 
10 \int_X  \left|  \beta \otimes \nabla ( fW^+) \right|_h  d\mu_h \\
&& \qquad\qquad  \qquad   + 
40 \int_X  \left|  \beta \otimes \beta \otimes fW^+ \right|_h  d\mu_h \\
&\leq & 2\int_X  f \left|  \nabla^* \nabla W^+\right|_h  d\mu_h \\
&& \qquad 
+  22 \int_X  |\nabla f|_h |\nabla W^+|_h  d\mu_h \\
&& \qquad\qquad  \qquad   + 
10 \int_X  |\nabla\nabla f|_h |W^+|_h  d\mu_h \\
&& \qquad\qquad\qquad  \qquad   + 
50 \int_X  f^{-1} |\nabla f|^2_h |W^+|_h  d\mu_h \\
&\leq & C_3 \int_M \Big[  |\omega|^{-1/2}_h+ |\nabla |\omega|^{-1/2}_h|_h \\
&& \qquad  \qquad + |\omega|^{1/2}_h |\nabla |\omega|^{-1/2}_h|_h^2 + 
 |\nabla \nabla |\omega|^{-1/2}_h|_h\Big] d\mu_h\\
 &\leq & C_3 \int_M \Big[  |\omega|^{-1/2}_h+ \frac{1}{2}|\omega|^{-3/2}_h |\nabla \omega|_h \\
&& \qquad +
\frac{23}{4}               |\omega|^{-5/2}_h |\nabla \omega |^2_h 
+2 |\omega|^{-3/2}_h |\nabla \nabla \omega |_h\Big] d\mu_h\\
&\leq & C_4 \int_M  |\omega|^{-5/2}_h  d\mu_h\\
&<  & \infty , 
\end{eqnarray*}
where $C_3$ is a positive constant depending on $(M,h)$,
$C_4$ is a positive constant depending on
$(M,h, \omega)$, and where, as in the remainder of the paper,  
$\nabla$   denotes the Levi-Civita connection $\hat{\nabla}$ of $h$ when its relation to $h$ is clearly indicated by  a subscript. 
Here,  in the last step, we have used the fact that $|\omega|^{-5/2}$ is comparable,   near $Z=M-X$,  to 
  $\mathsf{r}^{-5/2}$ on $B^3 \times S^1$, where 
$\mathsf{r}=|\vec{x}|$ is the distance from the origin  in the unit ball $B^3 \subset \RR^3$,
and therefore has finite integral because
$$\int_{B^3} |\vec{x}|^{-5/2} dx^1 \wedge dx^2\wedge dx^3 = 4\pi   \int_0^1 \mathsf{r}^{-5/2}\mathsf{r}^{2}d \mathsf{r}=4\pi  \Big[2\sqrt{\mathsf{r}}\Big]_0^1  < \infty.$$
In much the same way, 
\begin{eqnarray*}
\int_X \left|\langle fW^+ , \nabla^*\nabla (\omega \otimes \omega )\rangle_g\right| d\mu_g \negthickspace&\leq&
2\sqrt{2} \int_X   f|W^+|_g  | \nabla^*\nabla \omega|_g  d\mu_g \\
&& \qquad \qquad   + 
2 \int_X   f|W^+|_g  | \nabla \omega|_g^2  d\mu_g 
\\
&\leq& 2^{3/2} \int_X f^3 |W^+|_h f^4\Big[ |\nabla^*\nabla \omega |_h + 2 |\nabla  \beta |_h  |\omega |_h \\
&& \qquad \qquad \qquad    + 4 |\beta|_h |\nabla \omega|_h + |\beta |^2 |\omega|_h 
\Big] f^{-4} d\mu_h \\
&& \qquad + 2 \int_X f^3 |W^+|_h f^6 \Big[ |\nabla \omega |_h^2  + 4 | \beta |_h  |\omega|_h  |\nabla \omega |_h
\\&&  \qquad \qquad  \qquad  \qquad\qquad  \qquad    + 4 | \beta |_h^2  |\omega|_h^2
\Big] f^{-4} d\mu_h \\
&\leq&  2^{3/2} \int_X  |W^+|_h \Big[ f^3 |\nabla^*\nabla \omega |_h + 2 f^2 | \nabla( f^{-1}\nabla f ) |_h  |\omega |_h \\
&& \qquad \qquad  \qquad      + 4 f^2 |\nabla f|_h |\nabla \omega|_h + f |\nabla f |^2 |\omega|_h 
\Big] d\mu_h \\
&& \qquad + 2 \int_X  |W^+|_h \Big[ f^5  |\nabla \omega |_h^2  + 4 f^4 | \nabla f |_h  |\omega|_h  |\nabla \omega |_h
\\&&  \qquad \qquad  \qquad  \qquad   \qquad  \qquad + 4f^3  | \nabla f |_h^2  |\omega|_h^2
\Big]  d\mu_h \\
&\leq& C_5 \int_X  \Big[ |\omega|_h^{-3/2} |\nabla^*\nabla \omega |_h +    |\omega|_h^{-2} |\nabla \omega |_h^2   +   \\
&& \qquad \qquad  \qquad     
 |\omega|_h^{-2} |\nabla \nabla \omega |_h^2 +   |\omega|_h^{-5/2} |\nabla \omega |_h^2\Big]  d\mu_h \\
&\leq & C_6 \int_M  |\omega|_h^{-5/2} d\mu_h 
\\ & < & \infty 
\end{eqnarray*}
where $C_5$ and $C_6$ are positive constants depending, respectively, on $(M,h)$ and  $(M,h, \omega)$.
Thus,  the integrands in \eqref{party} both belong to $L^1$, and our previous computation therefore shows that
their integrals on $X$ are not  merely both  defined, but are actually equal. 
\end{proof}
 
 Since we are thus entitled to carry out the desired integration-by-parts in the near-symplectic case, 
 \eqref{initio} therefore implies an interesting integral Weitzenb\"ock formula when  $h$ also satisfies $\delta W^+=0$.

\pagebreak 

\begin{prop} 
\label{steady}
Let $\omega$ be a near-symplectic self-dual harmonic $2$-form on  a compact oriented
Riemannian $4$-manifold $(M,h)$ with $\delta W^+=0$. Let $X=M-Z$ be the complement of the zero set
$Z$ of $\omega$,  set $f=2^{1/4}|\omega|_h^{-1/2}$ on $X$, and let $g=f^{-2}h$ be the almost-K\"ahler metric on 
$(X,\omega)$ obtained by conformally rescaling $h$ to make $|\omega|_g\equiv \sqrt{2}$. Then $g$ satisfies 
$$
\int_X \Big[ \langle W^+ , \nabla^*\nabla (\omega\otimes \omega )\rangle  + \frac{s}{2} W^+(\omega , \omega ) 
 - 6 |W^+(\omega)|^2+ 2 |W^+|^2 |\omega |^2 \Big] f~ d\mu_g =0,
$$ 
 both as a  Lebesgue integral and as an improper Riemann integral. 
\end{prop}
\begin{proof}
Contraction of \eqref{initio} with $\omega\otimes \omega$  tells us that
$$ \langle \nabla^*\nabla  (fW^+) , \omega\otimes \omega \rangle  + \frac{s}{2}f W^+(\omega , \omega ) 
 - 6 f|W^+(\omega)|^2+ 2 f|W^+|^2 |\omega |^2 =0$$
 on $(X,g)$, so integration certainly tells us that
 $$\negthickspace 
 \int_X \Big[  \langle \nabla^*\nabla  (fW^+) , \omega\otimes \omega \rangle  + \frac{s}{2}f W^+(\omega , \omega ) 
 - 6 f|W^+(\omega)|^2+ 2 f|W^+|^2 |\omega |^2  \Big] d\mu_g =0.$$
 However, because  the first term is $L^1$, equation \eqref{initio} tells us that the same is also  true of  the sum of the remaining terms, and 
  Lemma \ref{primed} therefore allows us to rewrite the above expression as 
 $$
 \int_X \Big[ \langle fW^+ , \nabla^*\nabla (\omega\otimes \omega )\rangle  + f\frac{s}{2} W^+(\omega , \omega ) 
 - 6 f|W^+(\omega)|^2+ 2 f |W^+|^2 |\omega |^2 \Big]  d\mu_g =0.
 $$
 Collecting the common of factor of $f$ now yields the desired result. 
\end{proof}

\section{Some Almost-K\"ahler Geometry} 
\label{demilo}

When  an oriented Riemannian manifold $(M,h)$  with $\delta W^+=0$ carries a near-symplectic self-dual harmonic $2$-form $\omega$, 
we saw in Proposition \ref{steady} that, if we set 
  $f=2^{1/4}|\omega|_h^{-1/2}$ on the open 
set $X$ where 
$\omega\neq 0$,  
 the conformally related almost-K\"ahler metric  $g=f^{-2}h$ then   satisfies an integral Weitzenb\"ock formula on $X$. 
 In order to exploit this effectively, we will  next need a universal   identity
previously pointed out in \cite{lebcake}: 

\begin{lem}
\label{redstone}
Any $4$-dimensional almost-K\"ahler manifold satisfies 
$$\langle W^+ , \nabla^*\nabla (\omega\otimes \omega )\rangle= [W^+(\omega , \omega )]^2 + 4 |W^+(\omega )|^2 - s W^+ (\omega , \omega )$$
at every point. 
\end{lem}
\begin{proof} First notice our the oriented Riemannian $4$-manifold $(X,g)$ satisfies
$$\Lambda^+\otimes \CC = \CC\omega \oplus K \oplus \overline{K},$$
where 
$K= \Lambda^{2,0}_J$ is the canonical line bundle of the almost-complex structure $J$ defined by $\omega = g(J\cdot , \cdot )$.
Locally choosing a unit section $\varphi$ of $K$, we thus have 
$$\nabla \omega = \alpha \otimes \varphi + \bar{\alpha} \otimes \bar{\varphi}$$
for a unique $1$-form $\alpha \in \Lambda^{1,0}_J$, since  $\nabla_{[a}\omega_{bc]}=0$ and $\omega^{bc}\nabla_a \omega_{bc}=  0$. 
If 
$$\circledast: \Lambda^+\times \Lambda^+\to \odot^2_0\Lambda^+$$
  denotes the symmetric trace-free product, we therefore have 
$$(\nabla_e \omega ) \circledast  (\nabla^e\omega )= 2|\alpha |^2 \varphi \circledast\bar{\varphi}  = -\frac{1}{4} |\nabla \omega|^2\omega \circledast \omega$$
and we thus deduce that 
\begin{eqnarray*} 
\langle W^+ , \nabla^*\nabla (\omega\otimes \omega )\rangle
& = &2W^+(\omega , \nabla^*\nabla \omega  ) - 2W^+(\nabla_e \omega , \nabla^e \omega ) \\
& = & 2 W^+(\omega , \nabla^*\nabla \omega  )  + \frac{1}{2}|\nabla \omega |^2 W^+(\omega , \omega )\\
& = & 2 W^+(\omega , 2W^+ ( \omega ) - \frac{s}{3} \omega  ) + 
\Big[ W^+(\omega , \omega ) -\frac{s}{3}\Big] W^+(\omega , \omega )\\
& = & -\frac{2}{3}s W^+ (\omega , \omega ) + 4 |W^+(\omega )|^2 
+ \Big[ W^+(\omega , \omega ) -\frac{s}{3}\Big] W^+(\omega , \omega )\\
& = &  [W^+(\omega , \omega )]^2 + 4 |W^+(\omega )|^2 - s W^+ (\omega , \omega )
 \end{eqnarray*}
 where we have used the Weitzenb\"ock formula 
 $$0= \nabla^* \nabla \omega - 2 W^+(\omega  ) + \frac{s}{3}\omega$$
for the harmonic self-dual $2$-form $\omega$, as well as the associated key identity 
\begin{equation}
\label{cornerstone}
\frac{1}{2} |\nabla \omega |^2 =  W^+(\omega , \omega ) - \frac{s}{3}
\end{equation}
resulting from the fact that  $|\omega |^2\equiv 2$.
 \end{proof}

In conjunction with Proposition \ref{steady}, this  now yields the following:

\begin{thm} 
\label{go}
Let $\omega$ be a near-symplectic self-dual harmonic $2$-form on  a compact oriented
Riemannian $4$-manifold $(M,h)$ with $\delta W^+=0$. Let $X=M-Z$ be the complement of the zero set
$Z$ of $\omega$,  set $f=2^{1/4}|\omega|_h^{-1/2}$ on $X$, and let $g=f^{-2}h$ be the almost-K\"ahler metric on 
$(X,\omega)$ obtained by conformally rescaling $h$ to make $|\omega|_g\equiv \sqrt{2}$. Then the almost-K\"ahler metric $g$ satisfies 
\begin{equation}
\label{walla}
\int_X \left[ 8  \left(|W^+|^2 - \frac{1}{2} |W^+(\omega )^\perp|^2\right)-sW^+(\omega , \omega ) \right] f~d\mu_g =0, 
\end{equation}
where $s$ is the scalar curvature of $g$, and where $W^+(\omega )^\perp$ denotes the orthogonal projection of $W^+(\omega)$ to the orthogonal complement of 
$\omega \in \Lambda^+$. Moreover,  the integrand  belongs to $L^1$, so the statement holds 
whether left-hand-side is is construed as  a Lebesgue  integral or  as an improper Riemann integral. 
\end{thm}
\begin{proof}
Combining Proposition \ref{steady} with Lemma \ref{redstone}, 
  we    have 
\begin{eqnarray*} 0&=& 
 \int_X \Big[ \langle W^+ , \nabla^*\nabla (\omega\otimes \omega )\rangle  + \frac{s}{2} W^+(\omega , \omega ) 
 - 6 |W^+(\omega)|^2+ 2 |W^+|^2 |\omega |^2 \Big] f~ d\mu  \\
&=&  \int_X \Big[ \Big([W^+(\omega , \omega )]^2 + 4 |W^+(\omega )|^2 - s W^+ (\omega , \omega )\Big)
 \\&& \hphantom{\int_X \Big[ \Big([W^+(\omega , \omega )]^2 + 4 |W^+ } 
 + \frac{s}{2} W^+(\omega , \omega ) 
 - 6 |W^+(\omega)|^2+ 4 |W^+|^2 \Big]  f~d\mu  
 \\
&=&  \int_X \Big[ [W^+(\omega , \omega )]^2    - \frac{s}{2} W^+(\omega , \omega ) 
 - 2 |W^+(\omega)|^2+ 4 |W^+|^2 \Big]  f~d\mu ~. 
 \end{eqnarray*}
 Since  $|W^+(\omega)^\perp|^2 = |W^+(\omega)|^2 - \frac{1}{2}[W^+(\omega, \omega)]^2$, 
multiplication  by $2$ thus yields the desired formula \eqref{walla}. Moreover,  this calculation shows
that the integrand is the sum of two $L^1$ functions, and is therefore itself $L^1$ by the triangle inequality. 
\end{proof}

Next, we prove  a refinement of a point-wise inequality used in  \cite{lebcake}: 
\begin{lem}
\label{dominion}
Any $4$-dimensional almost-K\"ahler manifold satisfies 
$$
|W^+|^2 - \frac{1}{2} |W^+(\omega )^\perp|^2\geq \frac{3}{8} \left[ W^+ (\omega , \omega ) \right]^2+\frac{1}{2} |W^+(\omega )^\perp|^2
$$
at every point.
\end{lem} 
\begin{proof}
If $A=[A_{jk}]$ is any symmetric  trace-free $3\times 3$ matrix, the fact that $A_{33}= -(A_{11}+A_{22})$ implies that  
$$\sum_{jk} A_{jk}^2 \geq 2A_{21}^2 +  A_{11}^2+A_{22}^2+(A_{11}+A_{22})^2 = 2A_{21}^2 + \frac{3}{2}A_{11}^2 + 2 (\frac{A_{11}}{2}+A_{22})^2$$
and we therefore conclude that 
$$|A|^2 \geq  2 A_{21}^2 + \frac{3}{2} A_{11}^2 .$$

If we now let $A$ represent $W^+:\Lambda^+\to \Lambda^+$ with respect to an 
orthogonal basis $\mathfrak{e}_1, \mathfrak{e}_2, \mathfrak{e}_3$ for $\Lambda^+$ such that  $\omega = \sqrt{2}\mathfrak{e}_1$ and 
$W^+(\omega )^\perp \propto \mathfrak{e}_2$, this inequality becomes
$$
|W^+|^2 \geq |W^+(\omega)^\perp|^2 + \frac{3}{8} \left[ W^+ (\omega , \omega )\right]^2
$$
and  subtracting $\frac{1}{2}|W^+(\omega)^\perp|^2$ from both sides therefore proves the claim. 
 \end{proof} 

This now yields  a  key inequality: 

\begin{lem}
Let $(M,h)$, $\omega$, $X$, $g$ and $f$ be as in Theorem \ref{go}. Then the almost-K\"ahler metric $g=f^{-2}h$ satisfies 
\begin{equation}
0\geq \int_{X}  \left[  W^+ (\omega , \omega ) |\nabla \omega |^2 + \frac{8}{3}  |W^+(\omega )^\perp|^2 \right]f~d\mu_g ,
\label{punch}
\end{equation}
in the sense   the Lebesgue integral on the right is well-defined  and belongs to $[-\infty ,0]$.
\end{lem}

\begin{proof}
Theorem \ref{go}  tells us  that 
$$0=\int_{X} \left[ 8  \left(|W^+|^2 - \frac{1}{2} |W^+(\omega )^\perp|^2\right)-sW^+(\omega , \omega ) \right] f~d\mu_g$$
and that the  positive and negative parts of the integrand are both $L^1$ functions. 
The pointwise inequality of integrands  provided  by Lemma \ref{dominion} therefore impiies that 
$$
0\geq 
\int_{X} \left[  3  \left[ W^+ (\omega , \omega ) \right]^2 -sW^+(\omega , \omega )+4 |W^+(\omega )^\perp|^2\right] f~d\mu_g
$$
in the Lebesgue sense. After dividing by $3$, we can then re-express this as 
 \begin{equation}
\label{slap}
0\geq\int_{X} \left[ W^+(\omega , \omega ) \left( W^+ (\omega , \omega ) -\frac{s}{3}\right) + \frac{4}{3}  |W^+(\omega )^\perp|^2 \right]~ f~d\mu_g .
\end{equation}
However, \eqref{cornerstone} tells us that  $W^+(\omega , \omega ) - \frac{s}{3}=\frac{1}{2} |\nabla \omega |^2$ for any almost-K\"ahler $4$-manifold.
Making this substitution in \eqref{slap} and then multiplying by $2$ thus yields the desired inequality  \eqref{punch}.
\end{proof} 

In the special case where $(M,h,\omega)$ satisfies the conformally invariant condition $W^+(\omega , \omega ) \geq 0$,
we thus obtain the following: 

\begin{prop} 
\label{stay}
Let $(M,h)$ be a compact oriented
Riemannian $4$-manifold that satisfies  $\delta W^+=0$, and suppose that 
$\omega$ is  a near-symplectic self-dual harmonic $2$-form on   $(M,h)$ that satisfies  $W^+(\omega , \omega) \geq 0$. 
Let $X$, $g$, and $f$ be as in Theorem \ref{go}. Then the almost-K\"ahler manifold $(X,g,\omega)$ satisfies 
\begin{equation}
\int_{X}  \left[  W^+ (\omega , \omega ) |\nabla \omega |^2 + \frac{8}{3}  |W^+(\omega )^\perp|^2 \right]f~d\mu_g  =0,
\label{judy}
\end{equation}
both as a Lebesgue and as an improper Riemann integral.  
\end{prop}
\begin{proof}
The added assumption that $W^+(\omega , \omega) \geq 0$ obviously implies 
$$\int_{X}  \left[  W^+ (\omega , \omega ) |\nabla \omega |^2 + \frac{8}{3}  |W^+(\omega )^\perp|^2 \right]f~d\mu_g \geq 0$$
as an extended real number, because the integrand is now  non-negative. But in conjunction with  \eqref{punch}, this  immediately that 
$$
\int_{X}  \left[  W^+ (\omega , \omega ) |\nabla \omega |^2 + \frac{8}{3}  |W^+(\omega )^\perp|^2 \right]f~d\mu_g = 0
$$
as a Lebesgue integral. Since the integrand is also moreover $L^1$, the integral also necessarily vanishes 
as an improper Riemann integral. 
\end{proof}

This  very  strong statement  now has even stronger     consequences: 

\begin{prop} 
\label{claret}
Let $M$, $h$, $\omega$, $X$, $g$ and $f$ be as in Proposition \ref{stay}. 
 Then either $g$ is a K\"ahler metric on $X$ whose scalar curvature is  given by $s=\mathbf{c}/f$ for some  constant $\mathbf{c} >  0$, 
 or else $g$ satisfies $W^+\equiv 0$, and so is an anti-self-dual metric. \end{prop}
\begin{proof} 
Since  $f > 0$ by construction, and since  $W^+(\omega , \omega ) \geq 0$ by assumption, both terms in the integrand of
\eqref{judy} must vanish identically. We thus have 
\begin{equation}
\label{arrival}
W^+(\omega , \omega ) |\nabla \omega |^2 = 0 \qquad \mbox{and} \qquad W^+(\omega )^\perp = 0
\end{equation}
at every point of  $X$. 
In particular, 
$\nabla \omega=0$ wherever $W^+(\omega , \omega )\neq 0$.
If $\mathscr{V}\subset X$ is the open subset where $W^+(\omega , \omega )\neq 0$, the restriction of $g$ to $\mathscr{V}$  is therefore   K\"ahler.
On the other hand,  since $h=f^2g$ satisfies  $\delta W^+=0$,   conformal invariance of this equation tells us that 
$g$ satisfies $\delta (fW^+)=0$, as previously noted. On $(\mathscr{V},g)$ we therefore have
\begin{eqnarray*}
0&=&\omega^{ab}\omega^{cd} \nabla^e(fW^+_{ebcd})=  \nabla^e(fW^+_{ebcd}\omega^{ab}\omega^{cd})\\&= &\nabla^e(f \frac{s}{3}\omega_{eb}\omega^{ab})=
\frac{1}{3} \nabla^e(f s~\delta_e^b)= \frac{1}{3} \nabla_b(f s)= \nabla_b[fW^+(\omega , \omega)],
\end{eqnarray*}
since at each point of  any K\"ahler manifold of real dimension $4$, 
the K\"ahler form $\omega$ is an eigenvector of $W^+: \Lambda^+\to \Lambda^+$,
with eigenvalue  one-sixth of the scalar curvature $s$.
 This shows  that  $d[fW^+(\omega, \omega )]=0$ on $\mathscr{V}$, 
 and therefore,  by continuity,  on the closure $\overline{\mathscr{V}}$ of $\mathscr{V}$, too. On the other hand, since our definition of $\mathscr{V}$ guarantees that 
 $fW^+(\omega, \omega ) \equiv 0$ on the open set $X-\overline{\mathscr{V}}$,  we also have 
$d[fW^+(\omega, \omega )]=0$ on  $X-\overline{\mathscr{V}}$. It follows that $d[fW^+(\omega, \omega )]=0$ on all of $X$. Since $X$ is connected, and since $fW^+(\omega, \omega )\geq 0$, 
we therefore  conclude
  that $fW^+(\omega, \omega )=\mathbf{c}/3$ for
some non-negative constant $\mathbf{c}\geq 0$. 

If $\mathbf{c}>0$, $\mathscr{V}=X$, and it follows that $(X,g,\omega)$ is a K\"ahler manifold, with 
$$s=3 W^+(\omega, \omega ) = \frac{\mathbf{c}}{f}.$$

Otherwise,  $\mathbf{c}=0$, and we have $W^+(\omega, \omega ) \equiv 0$. On the other hand,  \eqref{arrival} also tells us  that  
$W^+(\omega )^\perp \equiv 0$
on $X$. Substituting these two facts  into \eqref{walla} then yields
$$
\int_{X}  | W^+ |^2 f~d\mu_g  =0.
$$
Thus, when $\mathbf{c}=0$, we conclude that $W^+\equiv 0$, and $g$ is therefore anti-self-dual in this remaining case, exactly as claimed.
\end{proof}

Sharpening these conclusions now supplies  our mainspring  result: 

\begin{thm} \label{clarify}
Let  $(M,h)$ be a compact oriented
Riemannian $4$-manifold with $\delta W^+=0$ that admits a near-symplectic self-dual harmonic $2$-form $\omega$ 
such that 
$$W^+(\omega , \omega ) \geq 0.$$
Then either $h$ satisfies $W^+\equiv 0$, and so is anti-self-dual, or else $W^+ (\omega , \omega )$ is   
everywhere positive,  and 
 $M$ admits a global K\"ahler metric $g$  with scalar curvature $s > 0$  
 such that $h=s^{-2}g$. 
\end{thm}
\begin{proof}
If $(X,g)$ satisfies $W^+\equiv 0$, the conformal invariance of this condition implies that $(X,h)$ satisfies
$W^+\equiv 0$, too. But since $X\subset M$ is dense, it then follows by continuity that $h$ satisfies $W^+\equiv 0$ on all 
of $M$. Thus, $(M,h)$ must be a compact  anti-self-dual manifold in this case. 

Otherwise,  $W^+\not\equiv 0$, and Proposition \ref{claret}  then  guarantees that  
 $g=f^{-2}h$ must  be a    K\"ahler metric on $X=M-Z$,  with K\"ahler form $\omega$ and 
$$3 W^+(\omega, \omega ) =s = \mathbf{c}f^{-1}$$
for some positive constant $\mathbf{c}$. However, since $h=f^2 g$, we also have  
$$[W^+(\omega, \omega ) ]_h = f^{-6}[W^+(\omega, \omega ) ]_g,$$
and it therefore follows that 
$$[W^+(\omega, \omega ) ]_h = \frac{\mathbf{c}}{3}~ f^{-7}.$$
But since $f=2^{1/4}|\omega|_h^{-1/2}$ by construction, this means that 
\begin{equation}
\label{buffo}
[W^+(\omega, \omega ) ]_h= \mathbf{b} ~|\omega|_h^{7/2}
\end{equation}
on $X$, where $\mathbf{b}  = \sqrt[4]{2} \mathbf{c}/12$ is a 
 a positive constant. 
 However, since $g$ is K\"ahler, with positive scalar curvature and K\"ahler form $\omega$, 
 $W^+$ has a repeated  negative eigenvalue at every point of $X$, and $\omega$ everywhere  belongs to the positive eigenspace. This implies  that 
 $$W^+(\omega , \omega ) = \sqrt{\frac{2}{3}} |W^+| |\omega|^2$$
 at every point of $X$, whether  for $g$ or for $h$. Thus \eqref{buffo} implies that 
 \begin{equation}
\label{comedia}
|W^+|_h = \mathbf{a} |\omega|_h^{3/2}
\end{equation}
everywhere on  $X$, where $\mathbf{a}=  \sqrt{\frac{3}{2}}\mathbf{b}$ is another positive constant. However, since $X\subset M$ is dense, and because the two sides
are both continuous functions, it then follows that \eqref{comedia} actually
holds on all of $M$. Now notice that this implies  that $|W^+|$ is everywhere differentiable; moreover, 
$W^+$ must vanish to first order along $Z$; thus, 
$\nabla W^+ =0$ at every point of  $Z$, where $\nabla$ denotes the Levi-Civita connection of $h$. Next, notice that \eqref{comedia} also implies  that 
$$|d|W^+|_h|_h = \frac{3}{2} \mathbf{a} |\omega|_h^{1/2} |d|\omega|_h|_h$$
on $X=M-Z$. Since the near-symplectic of $\omega$ moreover guarantees that 
$|d|\omega|_h|$ is bounded away from zero near $Z$, we therefore have
$$|d|W^+|_h|_h \geq \mathbf{A} |\omega|_h^{1/2}$$
on some neighborhood $\mathscr{U}$ of $Z$, where 
$\mathbf{A} :=  \frac{3}{2} \mathbf{a}~\inf_{\mathscr{U}-Z} |d|\omega|_h|_h$ is another positive constant. 
By the Kato inequality, we therefore have
$$|\nabla W^+|_h \geq  \mathbf{A} |\omega|_h^{1/2}$$
on $\mathscr{U}$.
 But since $h$ has been assumed throughout  to be a $C^4$ metric, $\nabla W^+$ is a differentiable tensor field,
and we have moreover previously  observed that this field vanishes along $Z$. It thus follows that 
$|\nabla W^+|_h$  is a Lipschitz function that vanishes along $Z$. But since $\omega$ is near-symplectic,  $|\omega|_h$ is commensurate with the distance
from  $Z$ in a  small enough neighborhood $\mathscr{U}\supset Z$,  and we 
must therefore have $\mathbf{B} |\omega|_h > |\nabla W^+|_h$ on a sufficiently small neighborhood $\mathscr{U}$ of $Z$,  for some positive constant $\mathbf{B}$.
But this then says that 
$$\mathbf{B}  |\omega|_h > \mathbf{A} |\omega|_h^{1/2}$$
on $\mathscr{U}$, and so implies
that 
$$ |\omega|_h > \frac{\mathbf{A}^2}{\mathbf{B}^2} > 0$$
on $\mathscr{U}-Z$. But since $X-(\mathscr{U}-Z)=M-\mathscr{U}$ is compact, and since $\omega\neq 0$ on $X$,   this implies 
 that $|\omega_h|$ is uniformly bounded away from zero on all of $X$. But since $X$ is dense in $M$, it therefore follows by continuity  that $|\omega|_h$
 is bounded away from zero on all of $M$. 
Since $Z$ is by definition  the zero set of $\omega$, we are therefore forced to conclude that $Z=\varnothing$.

Thus,   $g$ is a globally-defined K\"ahler metric  with scalar curvature $s> 0$
such that  $h=f^{2}g=\mathbf{c}^2s^{-2}g$ on all of  $M$. By now replacing $\omega$ with $\mathbf{c}^{-2/3}\omega$ 
and thus replacing 
$g$ with $\mathbf{c}^{-2/3}g$, we can 
 moreover now arrange for $h$ to simply be given by $s^{-2}g$, as promised. 
\end{proof}

\medskip 

This tells us quite a bit about the  $4$-manifolds that   carry metrics $h$ of the type covered  by Theorem \ref{clarify}.
Indeed \cite{bes,derd}, if $(M,J,g)$ is a compact K\"ahler surface of scalar curvature $s> 0$, then  
$h=s^{-2}g$ is a metric on $M$ with $\delta W^+=0$,  and with $W^+(\omega , \omega ) > 0$ 
for     the K\"ahler form $\omega$ of $g$.  On the other hand, 
if a  compact complex surface $(M,J)$  admits 
K\"ahler metrics $g$  with $s> 0$, it   is necessarily rational or ruled \cite{yauruled}. Conversely, any rational or ruled surface 
has arbitrarily small deformations that admit such metrics \cite{hitpos,mhsung}. Up to oriented diffeomorphism, we can therefore give a complete
list of the $4$-manifolds that admit solutions of this first type: they are $\CP_2$, $(\Sigma^2 \times S^2) \# k \overline{\CP}_2$, 
 and $\Sigma^2 \rtimes S^2$, where  $\Sigma$ is any compact orientable surface, $k$ is any non-negative integer, and 
$\Sigma^2 \rtimes S^2$ is the non-trivial oriented $2$-sphere bundle over $\Sigma$.  The moduli space of solutions on any  of these manifolds 
is moreover infinite-dimensional.

 The other class of   solutions allowed by Theorem \ref{clarify} is rather different, both because the moduli spaces
 of solutions are always finite dimensional, and because the near-symplectic self-dual harmonic $2$-form $\omega$
 is allowed to  have non-empty zero set. Of course, a vast menagerie of smooth compact oriented $4$-manifolds with $b_+\neq 0$ 
 is known to admit anti-self-dual metrics \cite{lebsing2,tasd}, but little is known about when 
 their self-dual harmonic $2$-forms $\omega$ are near-symplectic. 
  There certainly  are many examples with nowhere-zero $\omega$
 that are not conformally K\"ahler \cite{inyoungagag}, but there are also related explicit families  \cite{bisleb} with $b_+=1$ where the
 self-dual harmonic $2$-form $\omega$ transmutes  from being nowhere-zero to having non-empty zero locus.
For the latter  explicit anti-self-dual manifolds, it seems likely that the self-dual harmonic $2$-form $\omega$ is usually  
near-symplectic, but this is equivalent to  the non-degeneracy of all  critical points for a preferred harmonic function on  a 
quasi-Fuchsian hyperbolic $3$-manifold associated with the solution. Perhaps some interested reader will decide that
this  tractable-looking open problem   merits   careful  investigation!

\section{The Main Theorems}
\label{demedici}

With the results of \S \ref{demilo} in hand, we are now ready to prove our main theorems.

\medskip

\noindent 
{\sf Proof of Theorem \ref{diamond}.}  If $(M,h)$ is an  oriented  $4$-dimensional Einstein manifold, 
the second Bianchi identity implies that  $\delta W^+=0$.
If $(M,h)$ is  moreover compact, connected, and  admits a near-symplectic  self-dual 
harmonic $2$-form $\omega$ such that $W^+(\omega , \omega )\geq 0$, the conclusions of
Theorem \ref{clarify} then apply. Thus, if $W^+(\omega , \omega )> 0$ at some point, we know that $W^+ \not\equiv 0$, and 
 Theorem \ref{clarify}  then tells us  that $W^+(\omega , \omega )> 0$ everywhere, and  $h=s^{-2}g$
for some globally-defined K\"ahler metric $g$ on $M$ with scalar curvature $s > 0$. 
However, any $4$-dimensional Einstein metric is Bach-flat, and, because  this is a conformally 
invariant condition, the K\"ahler metric $g$ must therefore be Bach-flat, too. In particular, 
this implies \cite{chenlebweb,derd} that $g$ is an extremal K\"ahler metric. Moreover,
one can also show \cite{lebhem} that the complex structure associated with any such $g$ has $c_1>0$,
and it therefore follows  that  $M$ is necessarily diffeomorphic to a del Pezzo surface. Conversely,
each   del Pezzo diffeotype carries \cite{chenlebweb,sunspot,s,tian,ty}  an Einstein metric $h$ which can be written as $s^{-2}g$
for a suitable  extremal K\"ahler metric $g$ with scalar curvature $s > 0$. In fact, 
 $h$ is actually K\"ahler-Einstein in most cases, the only exceptions being when   $M$ is diffeomorphic to  $\CP_2\#  \overline{\CP}_2$ or 
$\CP_2\# 2 \overline{\CP}_2$.
 \hfill $\Box$

\medskip 

For each del Pezzo diffeotype,  the moduli space of all  Einstein metrics $h$ 
with  $W^+(\omega , \omega ) > 0$ is actually  connected  \cite{lebcake}. Moreover, it follows from    \cite[Theorem A]{lebuniq} 
and a modicum of   elementary Seiberg-Witten theory \cite[Theorem 3]{lmo}  that, 
for each del Pezzo $M$, this moduli space exactly coincides with the moduli space of 
all conformally K\"ahler, Einstein metrics.

\bigskip

\noindent 
{\sf Proof of Theorem \ref{pearl}.}  If $(M^4,h)$ is a  compact oriented $\lambda  \geq 0$ Einstein manifold that  carries a near-symplectic 
self-dual harmonic  
$\omega$ with $W^+(\omega , \omega )\geq 0$, then Theorem \ref{clarify} tells us that 
either $W^+(\omega , \omega ) > 0$ everywhere, or else $W^+\equiv 0$. 
Since the former case is covered by Theorem \ref{diamond}, we may therefore assume that 
$W^+\equiv 0$. However, by the Weitzenb\"ock  formula for the Hodge Laplacian, 
 the non-trivial self-dual harmonic $2$-form $\omega$ satisfies 
 $$0 = \nabla^*\nabla \omega - 2 W^+(\omega ) + \frac{s}{3} \omega$$
 and since $W^+=0$ and $s = 4\lambda \geq 0$ in our case, taking the inner product with $\omega$
 and integrating yields
 $$0 = \int_M \left[ |\nabla \omega |^2 + \frac{4\lambda}{3} |\omega |^2\right] d\mu_h.$$
 We therefore conclude that $\nabla \omega =0$ and $\lambda =0$, so that  $(M^4,h)$ is necessarily  Ricci-flat
 and  K\"ahler. Thus, after multiplying $\omega$ by a positive constant  if necessary in order to give it constant length $|\omega |_h \equiv \sqrt{2}$, 
 we see that $(M,h)$ carries an integrable, metric compatible  almost-complex structure $J$  such that
 $\omega = h(J\cdot , \cdot )$. Moreover,  since the K\"ahler metric $h$  is Ricci-flat, the 
 canonical line bundle $K$ of $(M,J)$ is flat, and $c_1(M,J)$ must therefore be a torsion class.
 The  Kodaira classification of complex surfaces  \cite{bpv,GH}  therefore tells us that $(M,J)$ must be a
   $K3$ surface, an
 Enriques surface, an Abelian surface, or a hyper-elliptic surface. 
 Conversely, Yau's solution of the Calabi conjecture \cite{yauma} tells us that each complex surface
 of one of these types carries a unique Ricci-flat K\"ahler metric in each K\"ahler class,
 and every such Calabi-Yau metric satisfies $W^+\equiv 0$.
 \hfill $\Box$
 
 \medskip
It is worth pointing out that the moduli space of Ricci-flat K\"ahler metrics is connected. Indeed, since the K\"ahler cone
is contractible for each complex structure, Yau's theorem reduces this statement to the known 
fact \cite{bpv} that all the $c_1^\RR=0$ complex structures on these $4$-manifolds are swept out by a 
single connected family.

\bigskip

Finally, let us observe that the near-symplectic hypothesis  is absolutely essential for Theorem \ref{diamond}:
\medskip 

\noindent 
{\sf Proof of Theorem \ref{opal}.} Let $(M,J,h)$ be a K\"ahler-Einstein metric
with $\lambda < 0$ on a compact complex surface $(M,J)$ with $p_g(M):=h^{2,0}(M)\neq 0$.
(For example, we could take $(M,J)$ to be a smooth quintic hypersurface in $\CP_3$,
so that $c_1(M) < 0$ and  $p_g(M)=4$, and let $h$ be the K\"ahler-Einstein metric
whose existence is guaranteed by the Aubin-Yau theorem \cite{aubin,yau}.)
Now recall that  the self-dual Weyl curvature  $W^+ : \Lambda^+\to \Lambda^+$ of any K\"ahler surface $(M^4,J,g)$ takes the form 
$$
\left[\begin{array}{ccc} -\frac{s}{12} &  &  \\ & -\frac{s}{12} &  \\ &  &  \frac{s}{6}\end{array}\right]
$$
 in any orthonormal  basis $\mathfrak{e}_1$, $\mathfrak{e}_2$, $\mathfrak{e}_3$
 for $\Lambda^+$ in  which $\mathfrak{e}_3$ is a multiple of the K\"ahler form, where $s$ is the scalar curvature. 
 Rather than taking $\omega$ to be the K\"ahler form, we now  instead take $\omega = \Re e ~(\varphi)$
 for some holomorphic $2$-form $\varphi\not\equiv 0$, on $(M,J)$. Of course,  the existence of such a $\varphi$  is 
 guaranteed by our assumption that $h^{2,0}\neq 0$. Notice that $\varphi$ is automatically self-dual  and harmonic
 as a consequence of standard K\"ahler identities, and that the same is therefore automatically true of its real part $\omega$.

However, since  $\omega \in \Re e ~\Lambda^{2,0}$ 
 is everywhere point-wise orthogonal to the K\"ahler form, 
 we now see that 
 $$W^+(\omega, \omega ) =  -\frac{s}{12}|\omega|^2 =  \frac{|\lambda |}{3}|\omega|^2\geq 0,$$
 since the Einstein constant  $\lambda$ of $h$ is assumed to be negative,
 Moreover, since $\omega \not\equiv 0$, this non-negative expression  is somewhere
 positive. On the other hand, the canonical line bundle of $(M,J)$ is non-trivial, because $c_1 (K)= - c_1 > 0$,
 so $\varphi$, and therefore $\omega$, must vanish along some non-empty holomorphic curve $\Sigma \subset M$. 
Thus, $W^+(\omega, \omega )$ vanishes somewhere, and  the conclusion of Theorem \ref{diamond}
therefore  fails for this class of examples. \hfill $\Box$

\bigskip

Of course, in light of  counter-examples like those detailed  in the proof of Theorem \ref{opal}, it is important to explain exactly where  the proof of Theorem \ref{diamond} 
 breaks down when $\omega$ is not  near-symplectic. In fact, 
the key failure occurs at the very beginning of our chain of reasoning, when Lemma \ref{primed} is deduced  from Lemma \ref{ready}.
Recall that Lemma \ref{ready} tells us that the boundary terms arising from integration by parts
have size 
$\sim \epsilon^{-3/2} \Vol^{(3)} ( \partial X_\epsilon , h)$,
where $\partial X_\epsilon$ is the hypersurface where $|\omega|_h = \epsilon$. 
In the near-symplectic case, $\Vol^{(3)} ( \partial X_\epsilon , h)\sim \epsilon^{2}$, so 
the boundary terms  are no worse than  $\epsilon^{1/2}$, and so vanish in the limit  as $\epsilon \to 0$. 
By contrast, in the above examples,  the zero locus $Z=\Sigma$ of $\omega$ has real codimension $2$, and we instead have
 $\Vol^{(3)} ( \partial X_\epsilon , h)\sim \epsilon$. This means  the boundary terms could in principle blow up as fast as   
 $\epsilon^{-1/2}$, and   so, in particular,  can then no longer   
be  expected to become   negligeable  as $\epsilon$ tends to zero.

\pagebreak

%

\vfill 

\noindent 
{\sc Department of Mathematics, Stony Brook University, Stony Brook, NY 11794-3651 USA} 

\medskip 

\noindent 
{e-mail:} claude@math.stonybrook.edu

\bigskip 

\noindent 
{\sc Keywords:} Einstein metric,  Weyl curvature,  Del Pezzo surface, K3 surface,  harmonic form, K\"ahler, almost-K\"ahler, nearly symplectic.

\bigskip 

\noindent 
{\sc MSC classification:}  53C25 (Primary),  14J26,  14J28, 32J15, 53C55, 53D05.

\end{document}